\documentclass[12pt, reqno]{amsart}
\usepackage{amsmath, amsthm, amscd, amsfonts, amssymb, graphicx, color, relsize}
\newtheorem{theorem}{Theorem}[section]
\newtheorem{lemma}[theorem]{Lemma}
\newtheorem{proposition}[theorem]{Proposition}

\newtheorem{corollary}[theorem]{Corollary}
\newtheorem{definition}[theorem]{Definition}

\newtheorem{remark}[theorem]{Remark}

\usepackage[bookmarksnumbered, colorlinks, plainpages]{hyperref}
\hypersetup{colorlinks=true,linkcolor=blue, anchorcolor=green, citecolor=cyan, urlcolor=red, filecolor=magenta, pdftoolbar=true}

\usepackage{comment}
\excludecomment{mysection}

\begin{document}
	
\title[Partial isometries]{Partial isometries in an absolute order unit space}
\author{Anil Kumar Karn and Amit kumar}
	
\address{School of Mathematical Sciences, National Institute of Science Education and Research, HBNI, Bhubaneswar, P.O. - Jatni, District - Khurda, Odisha - 752050, India.}

\email{\textcolor[rgb]{0.00,0.00,0.84}{anilkarn@niser.ac.in, amit.kumar@niser.ac.in}}

	

\begin{abstract}
In this paper, we extend the notion of orthogonality to the general elements of an absolute matrix order unit space and relate it to the orthogonality among positive elements. We introduce the notion of a partial isometry in an absolute matrix order unit space. As an application, we describe the comparison of order projections. We also discuss finiteness of order projections.
\end{abstract}

\thanks{The second author was financially supported by the Senior Research Fellowship of the University Grants Commission of India.}

\maketitle

\section{Introduction}

Order structure is one of the basic ingredient of C$^*$-algebra theory. If we keep Gelfand-Naimark theorem \cite{GN} as well as Kakutani theorem \cite{SK} in one place, we can deduce that the self-adjoint part of every commutative C$^*$-algebra is a Banach lattice besides having some other properties. As a contrast, Kadison's anti-lattice theorem \cite{Kad51} informs us that the self-adjoint part of a non-commutative C$^*$-algebra can not be a vector lattice. Thus the study of the order structure of a general C$^*$-algebra opens as an interesting area. The corresponding theory evolves a study of ordered vector spaces not having any vector lattice structure.

Initial significant contribution in this direction begins with Kadison's functional representation of operator (C$^*$-) algebras \cite{RVK} where he proved that the self adjoint part of any operator system can be represented as the space of affine continuous real valued functions on the state space of the given operator system. (An \emph{operator system} is a unital, self adjoint subspace of a unital C$^*$-algebra.) This seminal result turned out to be a benchmark and unfolded in the duality theory of ordered vector spaces. Early development of this theory can be found in the works of Bonsall \cite{B54,B56}, Edwards \cite{E64}, Ellis \cite{JE64}, Asimov \cite{LA,A74}, Ng \cite{KFN} besides many other people.

However, as a breakthrough, Choi and Effros \cite{CE77} characterized operator systems a \emph{matrix order unit spaces} (definition is given latter). This result set another benchmark for the study of the (order theoretic) functional analysis. Besides many others, the first author started working on the order theoretic aspect of C$^*$-algebras. He provided an order theoretic characterization of the algebraic orthogonality among positive elements of C$^*$-algebras. This characterization leads to notions of \emph{absolute order unit spaces} and \emph{absolute matrix order unit spaces} (definitions are given latter). Absolute order unit spaces can seen as a generalization of $AM$-spaces. (An $AM$-space is a Banach lattice satisfying an additional condition.) Here orthogonality plays an important role. 

In \cite{K18}, the first author developed a notion of spectral family of (order) projections for (self adjoint) elements of an absolute order unit space. In \cite{KK19}, the authors proved that a unital linear bijection between two absolute (matrix) order unit spaces is a (complete) isometry if and only if it is a (complete) absolute value preserving map. These results underscore the order theoretic proximity of absolute matrix order unit spaces to unital C$^*$-algebras. In this paper, we continue our study in the same direction. We provide an order theoretic definition of partial isometries suitable to an absolute matrix order unit spaces and extend the comparison theory of projections in a C$^*$-algebra to the comparison theory of order projections in an absolute matrix order unit spaces. The development of the paper is as follows. 

As per the definition of an absolutely ordered space, orthogonality is defined only among positive elements. In Section 2, we extend the notion of orthogonality to general elements. First, we extend it to general (self adjoint) elements of an absolutely ordered space and then to all elements in an absolutely matrix ordered space. We briefly discuss some of the properties of an absolute matrix order unit space. In section 3, we introduce partial isometries and other related notions order theoretically in an absolute matrix order unit space. We discuss some of the properties related to these notions. As an application of partial isometries, In Section 4, we propose a comparison of order projections in Section 4. This notion generalizes that of Murray-von Neumann equivalence of projections in a C$^*$-algebra. In this paper, we term it partial isometry equivalence. We also generalize unitary equivalence of projections in a C$^*$-algebra to absolute matrix order unit spaces. We discuss properties related to these two notions. In Section 5, we study finiteness of order projections.

\section{Extending orthogonality}

We begin with a definition. 
\begin{definition} \cite[Definition 3.4]{K18}
	A real ordered vector space $(V, V^+)$ is said to be an {\emph absolutely ordered space}, if there exists a mapping $\vert\cdot\vert: V \to V^+$ satisfying the following conditions:
	\begin{enumerate}
		\item $\vert v \vert = v$ if $v \in V^+$; 
		\item $\vert v \vert \pm v \in V^+$ for all $v \in V$; 
		\item $\vert k v \vert = \vert k \vert \vert v \vert$ for all $v \in V$ and $k \in \mathbb{R}$; 
		\item If  $\vert u - v \vert = u + v$ and $0 \le w \le v$, then $\vert u - w \vert = u + w$; and
		\item If $\vert u - v \vert = u + v$ and $\vert u - w \vert = u + w$, then $\vert u - \vert v \pm w \vert \vert = u + \vert v \pm w \vert$.
	\end{enumerate} 
	It is denoted by $(V, V^+, \vert\cdot\vert)$, or simply, by $V$, if there is no ambiguity.
\end{definition}

\begin{remark}\cite[Remark 2.3]{KK19}
	Let $(V, V^+, \vert\cdot\vert)$ be an absolutely ordered space.
	\begin{enumerate}
		\item  The cone $V^+$ is proper and generating. 		
		\item Let $u, v \in V$ be such that $\vert u - v\vert = u + v$. Then $u, v \in V^+$. For such a pair $u, v \in V^+$, we shall say that $u$ is \emph{orthogonal} to $v$ and denote it by $u \perp v$.
		\item We write, $v^+ := \frac{1}{2}(\vert v \vert + v)$ and $v^- := \frac{1}{2}(\vert v \vert - v)$. Then $v^+ \perp v^-, v = v^+ - v^-$ so that $\vert v \vert = v^+ + v^-$. This decomposition is unique. 
	\end{enumerate}
\end{remark}

 Let us also recall that the notion of absolute value can be characterized in terms of the natural orthogonality in $V^+$. In fact the following result (in a weaker form) was proved in \cite{K18}. The proof can be replicated. 
\begin{theorem} \cite[Theorem 3.1]{K18}
	Let $(V, V^+)$ be a real ordered vector space. Then $V$ is an absolutely ordered space if and only if there exists a binary relation $\perp$ in $V^+$ satisfying the following conditions: 
\begin{enumerate}
	\item $u \perp 0$ for all $u \in V^+$.
	\item If $u \perp v$, then $v \perp u$.
	\item If $u \perp v$, then $k u \perp k v$ for all $k \in \mathbb{R}$ with $k > 0$.
	\item For each $v \in V$, there exist a unique pair $v_1, v_2 \in V^+$ such that $v = v_1 - v_2$ with $v_1 \perp v_2$. 
	
	Let us write $\vert v \vert := v_1 + v_2$. 
	\item If $u, v, w \in V^+$, $u \perp v$ and $w \le v$, then $u \perp w$; and 
	\item If $u, v, w \in V^+$, $u \perp v$ and $u \perp w$, then $u \perp \vert v \pm w \vert$.
\end{enumerate} 
\end{theorem} 
In this section, we extend the notion of orthogonality to general elements. For $u, v \in V$, we say that $u$ is {\emph orthogonal} to $v$, we still write $u \perp v$, if $\vert \vert u \vert - \vert v \vert \vert = \vert u \vert + \vert v \vert$. (Or equivalently, $\vert u \vert \perp \vert v \vert$.) In the next result, we characterize orthogonality among general elements in an absolutely ordered space in terms of positive elements. Let us first recall that for each $v \in V$, $v^+, v^- \in V^+$ is the unique pair such that $v^+ \perp v^-$ and $v = v^+ - v^-$ (and $\vert v \vert = v^+ + v^-$). Here $2 v^+ := \vert v \vert + v$ and $2 v^- := \vert v \vert - v$.

\begin{proposition}\label{1}
	Let $(V, V^+, \vert\cdot\vert)$ be an absolutely ordered space and let $u, v \in V$. Then the following statements are equivalent: 
	\begin{enumerate}
		\item $u \perp v$;
		\item $u^+, u^-, v^+, v^-$ are mutually orthogonal;
		\item $\vert u \pm v \vert = \vert u \vert + \vert v \vert$.
	\end{enumerate}
\end{proposition}
\begin{proof} 
	(1) implies (2): Let $u \perp v$ or equivalently, $\vert u \vert \perp \vert v \vert$. As $0 \le u^+, u^- \le  \vert u \vert$ and $0 \le v^+, v^- \le \vert v \vert$, a repeated use of the definition yields that $u^+, u^-, v^+, v^-$ are mutually orthogonal.
	
	(2) implies (1): Let $u^+, u^-, v^+, v^-$ be mutually orthogonal. Then by the definition, we get $(u^+ + u^-) \perp (v^+ + v^-)$ for $\vert u^+ + u^- \vert = u^+ + u^-$ and $\vert v^+ + v^- \vert = v^+ + v^-$. That is, $\vert u \vert \perp \vert v \vert$ so that $u \perp v$.

	(2) implies (3): Again, let $u^+, u^-, v^+, v^-$ be mutually orthogonal. Then by the definition, once again, we have $(u^+ + v^+) \perp (u^- + v^-)$ and $(u^+ + v^-) \perp (u^- + v^+)$. Thus 
	\begin{eqnarray*}
		\vert u + v \vert &=& \vert u^+ - u^- + v^+ - v^- \vert \\ 
		&=& \vert (u^+ + v^+) - (u^- + v^-) \vert \\
		&=& u^+ + v^+ + u^- + v^- = \vert u \vert + \vert v \vert
	\end{eqnarray*}
	and 
	\begin{eqnarray*}
		\vert u - v \vert &=& \vert u^+ - u^- - v^+ + v^- \vert \\ 
		&=& \vert (u^+ + v^-) - (u^- + v^+) \vert \\
		&=& u^+ + v^- + u^- + v^+ = \vert u \vert + \vert v \vert.
	\end{eqnarray*}
	(3) implies (2): Finally, assume that $\vert u \pm v \vert = \vert u \vert + \vert v \vert$. Then as before, we may get that $(u^+ + v^+) \perp (u^- + v^-)$ and $(u^+ + v^-) \perp (u^- + v^+) $. Thus $u^+, u^-, v^+, v^-$ are mutually orthogonal.
\end{proof}
Now, we shall extend the notion of orthogonality to the general elements in an absolutely matrix ordered space. Recall that a \emph{matrix ordered space} is a $*$-vector space $V$ together with a sequence $\lbrace M_n(V)^+ \rbrace$ with $M_n(V)^+ \subset M_n(V)_{sa}$ for each $n \in \mathbb{N}$ satisfying the following conditions: 
\begin{enumerate}
	\item[(a)] $(M_n(V)_{sa}, M_n(V)^+)$ is a real ordered vector space, for each $n \in \mathbb{N}$; and  
	\item[(b)] $\alpha^* v \alpha \in M_m(V)^+$ for all $v \in M_n(V)^+$, $\alpha \in M_{n,m}$ and $n ,m \in \mathbb{N}$. 
\end{enumerate} 
It is denoted by $(V, \lbrace M_n(V)^+ \rbrace)$. 
\begin{definition}[\cite{KK19}, Definition 4.1]\label{2}
	Let $(V, \lbrace \ M_n(V)^+ \rbrace)$ be a matrix ordered space and assume that  $\vert\cdot\vert_{m,n}: M_{m,n}(V) \to M_n(V)^+$ for $m, n \in \mathbb{N}$. Let us write $\vert\cdot\vert_{n,n} = \vert\cdot\vert_n$ for every $n \in \mathbb{N}$. Then $\left(V, \lbrace M_n(V)^+ \rbrace, \lbrace \vert\cdot\vert_{m,n} \rbrace \right)$ is called an \emph{absolutely matrix ordered space}, if it satisfies the following conditions: 
	\begin{enumerate}
		\item[$1.$] For all $n \in \mathbb{N}$, $(M_n(V)_{sa}, M_n(V)^+, \vert\cdot\vert_n)$ is an absolutely ordered space;
		\item[$2.$] For $v \in M_{m,n}(V), \alpha \in M_{r,m}$ and $\beta \in M_{n,s},$ we have
		$$\vert \alpha v \beta \vert_{r,s} \leq \| \alpha \| \vert \vert v \vert_{m,n} \beta \vert_{n,s};$$
		\item[$3.$] For $v \in M_{m,n}(V)$ and $w \in M_{r,s}(V),$ we have
		$$\vert v \oplus w\vert_{m+r,n+s} = \vert v \vert_{m,n} \oplus \vert w \vert_{r,s}.$$
		Here $v \oplus w := \begin{bmatrix} v & 0 \\ 0 & w \end{bmatrix}$.
	\end{enumerate} 
Let $u, v \in M_{m,n}(V)$. We say that $u$ is \emph{orthogonal} to $v$, (we continue to write, $u \perp v$), if $\begin{bmatrix} 0 & u \\ u^* & 0 \end{bmatrix} \perp \begin{bmatrix} 0 & v \\ v^* & 0 \end{bmatrix}$ in $M_{m+n}(V)_{sa}$.
\end{definition}
We also recall the following properties. 
\begin{remark} \cite[Proposition 4.2]{KK19} \label{3}
	Let $(V, \lbrace M_n(V)^+ \rbrace, \lbrace \vert\cdot\vert_{m,n} \rbrace)$ be an absolutely matrix ordered space.  
	\begin{enumerate}
		\item If $\alpha \in M_{r,m}$ is an isometry i.e. $\alpha^* \alpha = I_m,$ then $\vert \alpha v \vert_{r,n} = \vert v \vert_{m,n}$ for any $v \in M_{m,n}(V).$
		\item If $v \in M_{m,n}(V),$ then $\left\vert \begin{bmatrix} 0_m & v \\ v^* & 0_n \end{bmatrix} \right\vert_{m+n} = \vert v^* \vert_{n,m} \oplus \vert v \vert_{m,n}.$
		\item $\begin{bmatrix} \vert v^* \vert_{n,m} & v \\ v^* & \vert v \vert_{m,n} \end{bmatrix} \in M_{m+n}(V)^+$ for any $v \in M_{m,n}.$
		\item $\vert v \vert_{m,n} = \left\vert \begin{bmatrix} v \\ 0 \end{bmatrix} \right\vert_{m+r,n}$ for any $v \in M_{m,n}$ and $r \in \mathbb{N}.$
		\item $\vert v \vert_{m,n} \oplus 0_s = \left\vert \begin{bmatrix} v & 0 \end{bmatrix} \right\vert_{m,n+s}$ for any $v \in M_{m,n}(V)$ and $s \in \mathbb{N}.$
	\end{enumerate} 
\end{remark}
\begin{remark}\label{4}
	Let $V$ be an absolutely matrix ordered space. 
	\begin{enumerate}
		\item Let $u_1,u_2\in M_m(V)^+$ and $v_1,v_2\in M_n(V)^+$ for some $m, n \in \mathbb{N}$. Then $u_1 \oplus v_1 \perp u_2 \oplus v_2$ in $M_{m+n}(V)^+$ if, and only if, $u_1 \perp u_2$ and $v_1 \perp v_2$ in $M_m(V)^+$ and $M_n(V)^+$ respectively. In fact, by \ref{2}(3), we have 
		$$\left \vert \begin{bmatrix} u_1 & 0 \\ 0 & v_1\end{bmatrix} - \begin{bmatrix} u_2 & 0 \\ 0 & v_2\end{bmatrix} \right \vert = \begin{bmatrix} \vert u_1-u_2\vert & 0 \\ 0 & \vert v_1-v_2\vert \end{bmatrix} =  \begin{bmatrix} u_1+u_2 & 0 \\ 0 & v_1+v_2 \end{bmatrix}.$$ 
		\item In particular, for $u,v\in M_{m,n}(V)$, we have $u\perp v$ if and only if $\vert u\vert_{m,n} \perp \vert v\vert_{m,n}$ and $\vert u^*\vert_{n,m} \perp \vert v^*\vert_{n,m}$ (using Remark \ref{3}(2)).
	\end{enumerate}
	
\end{remark}

\begin{proposition}\label{5}
	Let $V$ be an absolutely matrix ordered space and let $u,v\in M_{m,n}(V)$.
	\begin{enumerate}
		\item $u\perp v$ if and only if, $\vert u\pm v\vert_{m,n} = \vert u\vert_{m,n} +\vert v\vert_{m,n}$ and $\vert u^*\pm v^*\vert_{n,m} = \vert u^*\vert_{n,m} +\vert v^*\vert_{n,m}.$
		\item If $u \perp v$, then $\left \vert \begin{bmatrix} u \\ v\end{bmatrix}\right \vert_{2m,n} = \vert u\vert_{m,n} + \vert v\vert_{m,n}$ and $\left \vert \begin{bmatrix} u & v\end{bmatrix}\right \vert_{m,2n} = \vert u\vert_{m,n} \oplus \vert v\vert_{m,n}.$
	\end{enumerate} 
\end{proposition}

\begin{proof}
	(1). First, assume that $u \perp v$. Then by Proposition \ref{1}, we have 
	\begin{eqnarray*}
		\begin{bmatrix}\vert u^*\pm v^*\vert_{n,m} & 0 \\ 0 & \vert u\pm v\vert_{m,n} \end{bmatrix} &=& \left\vert\begin{bmatrix} 0 & u \\ u^* & 0\end{bmatrix} \pm \begin{bmatrix} 0 & v \\ v^* & 0\end{bmatrix}\right\vert_{m+n} \\
		&=& \left\vert\begin{bmatrix} 0 & u \\ u^* & 0\end{bmatrix}\right\vert_{m+n} + \left\vert\begin{bmatrix} 0 & v \\ v^* & 0\end{bmatrix}\right\vert_{m+n} \\
		&=& \begin{bmatrix} \vert u^*\vert_{n,m} & 0 \\ 0 & \vert u\vert_{m,n} \end{bmatrix} + \begin{bmatrix} \vert v^*\vert_{n,m} & 0 \\ 0 & \vert v\vert_{m,n} \end{bmatrix}.  
	\end{eqnarray*} 
	Thus $\vert u \pm v \vert_{m,n} = \vert u \vert_{m,n} + \vert v \vert_{m,n}$ and $\vert u^* \pm v^* \vert_{n,m} = \vert u^* \vert_{n,m} + \vert v^* \vert_{n,m}$. Now tracing back the proof, we can prove the converse also.

	(2). First, we observe that 
	$$\left \vert \begin{bmatrix} u \\ 0 \end{bmatrix}^*\right \vert_{n, 2m} = \begin{bmatrix} \vert u^*\vert_{n,m} & 0 \\ 0 & 0\end{bmatrix} \perp \begin{bmatrix} 0 & 0 \\ 0 & \vert v^*\vert_{n,m} \end{bmatrix}= \left \vert \begin{bmatrix} 0 \\ v \end{bmatrix}^*\right \vert_{n, 2m}$$ 
	and that 
	$$\left \vert \begin{bmatrix} u \\ 0 \end{bmatrix}\right \vert_{2m, n}= \vert u\vert_{m,n} \perp \vert v\vert_{m,n} = \left \vert \begin{bmatrix} 0 \\ v \end{bmatrix}\right \vert_{2m, n}.$$ 
	Thus by Remark \ref{4}(2) and (1), we get 
		$$\left \vert \begin{bmatrix} u \\ v\end{bmatrix} \right \vert_{2m,n} 
		= \left \vert \begin{bmatrix} u \\ 0 \end{bmatrix}\right \vert_{2m, n} + \left \vert \begin{bmatrix} 0 \\ v \end{bmatrix}\right \vert_{2m, n} 
		= \vert u\vert_{m,n} + \vert v\vert_{m,n}.$$
	
	Similarly, we have 
	\begin{eqnarray*}
		\left \vert \begin{bmatrix} u & v \end{bmatrix}\right \vert_{m,2n} 
		&=& \left \vert \begin{bmatrix} u & 0 \end{bmatrix}\right \vert_{m, 2n} + \left \vert \begin{bmatrix} 0 & v \end{bmatrix}\right \vert_{m, 2n} \\ 
		&=& \begin{bmatrix} \vert u \vert_{m,n} & 0 \\ 0 & 0 \end{bmatrix} + \begin{bmatrix} 0 & 0 \\ 0 & \vert v \vert_{m,n} \end{bmatrix} \\
		&=& \vert u\vert_{m,n} \oplus \vert v\vert_{m,n}
	\end{eqnarray*}
	for 
	$$\left \vert \begin{bmatrix} u & 0 \end{bmatrix}\right \vert_{m, 2n} = \begin{bmatrix} \vert u\vert_{m,n} & 0 \\ 0 & 0 \end{bmatrix} \perp \begin{bmatrix}0 & 0 \\ 0 & \vert v\vert_{m,n} \end{bmatrix} = \left \vert \begin{bmatrix} 0 & v \end{bmatrix}\right \vert_{m, 2n}$$ 
	and 
	$$\left \vert \begin{bmatrix} u & 0 \end{bmatrix}^*\right \vert_{2n, m} = \vert u^* \vert_{n,m} \perp \vert v^*\vert_{n,m} = \left \vert \begin{bmatrix} 0 & v \end{bmatrix}^*\right \vert_{2n, m}.$$
\end{proof}

\subsection{Absolute matrix order unit spaces}
Let $(V, \lbrace M_n(V)^+ \rbrace)$ be a matrix ordered space. Assume that $V^+$ is \emph{proper}, (that is, $V^+ \cap ( - V^+) = \lbrace 0 \rbrace$) and that $M_n(V)^+$ is \emph{Archimedean} for all $n \in \mathbb{N}$. ($V^+$ is said to be Archimedean, if given $u \in V_{sa}$, we have $v \in V^+$ whenever $v + k u \in V^+$ for all $k > 0$.)  Then  $(V, \lbrace M_n(V)^+ \rbrace, e)$ is said to be a matrix order unit space \cite{CE77}. If there is no ambiguity, we also denote it by  $(V, e)$ or even simply by $V$. In a matrix order unit space $(V, e)$, the order unit determines a matrix norm on $V$: 
$$\Vert v \Vert_n := \inf \lbrace k > 0: \begin{bmatrix} k e^n & v \\ v^* & k e^n \end{bmatrix} \in M_{2n}(V)^+ \rbrace \quad \textrm{for} ~ v \in M_n(V).$$ 
Let $u, v \in M_n(V)_{sa}$. We say that $u$ is $\infty$-\emph{orthogonal}, (we write $u \perp_{\infty} v$), if $\Vert u + k v \Vert_n = \max \lbrace \Vert u \Vert, \Vert k v \Vert_n \rbrace$ for all $k \in \mathbb{R}$. Recall that for $u, v \in M_n(V)^+ \setminus \lbrace 0 \rbrace$, we have $u \perp_{\infty} v$ if and only if $\Vert \Vert u \Vert_n^{- 1} u + \Vert v \Vert_n^{- 1} v \Vert_n = 1$ \cite[Theorem 3.3]{K14}. For $u, v \in M_n(V)^+ \setminus \lbrace 0 \rbrace$, we say that $u \perp_{\infty}^a v$, if $u_1 \perp_{\infty} v_1$ whenever $0 \le u_1 \le u$ and $0 \le v_1 \le v$. It was proved in \cite{K16, K18} that if $A$ is a C$^*$-algebra and if $a, b \in A^+ \setminus \lbrace 0 \rbrace$, then $a \perp_{\infty}^a b$ if and only if $a b = 0$ (that is, $a$ is algebraically orthogonal to $b$). 

\begin{definition}
	Let $(V, \lbrace M_n(V)^+ \rbrace, e)$ be a matrix order unit space and assume that 
	\begin{enumerate}
		\item[(a)] $\left(V, \lbrace M_n(V)^+ \rbrace, \lbrace \vert \cdot \vert_{m,n} \rbrace \right)$ is an absolutely matrix ordered space; and
		\item[(b)]$\perp = \perp_{\infty}^a$ on $M_n(V)^+$ for all $n \in \mathbb{N}.$ 
	\end{enumerate}
	Then $(V, \lbrace M_n(V)^+ \rbrace, \lbrace \vert\cdot\vert_{m,n} \rbrace, e)$ is called an \emph{absolute matrix order unit space} \cite[Definition 4.3]{KK19}. 
\end{definition}
\begin{proposition}\label{6}
	Let $V$ be an absolutely matrix ordered space and let $v \in M_{m,n}(V)$ for some $m,n \in \mathbb{N}.$ Then $v=0$ if and only if $\vert v\vert_{m,n}=0$ and $\vert v^*\vert_{n,m}=0.$ If $M_{m+n}(V)^+$ is Archimedean, then $v=0$ if and only if $\vert v\vert_{m,n}=0.$
\end{proposition}

\begin{proof}
	If $v=0$, then by the definition, we have $\vert v\vert_{m,n}=0$ and that $\vert v^* \vert_{n,m} = 0$. Conversely assume that $\vert v\vert_{m,n}=0$ and $\vert v^*\vert_{n,m}=0.$ By Remark \ref{3}(3), we deduce that $\pm \begin{bmatrix} 0 & v \\ v^* & 0 \end{bmatrix}\in M_{m+n}(V)^+.$ Since $M_{m+n}(V)^+$ is proper, we have $\begin{bmatrix} 0 & v \\ v^* & 0 \end{bmatrix}=0$ so that $v=0.$
	
	Now, we assume that $M_{m+n}(V)^+$ is Archimedean. If $\vert v\vert_{m,n}=0,$ then as above, we have $\begin{bmatrix} \vert v^*\vert_{n,m} & v \\ v^* & 0\end{bmatrix}\in M_{m+n}(V)^+.$ Let $k > 0$ be any real number and consider $\gamma = \begin{bmatrix} k I_m & 0 \\ 0 & k^{-1} I_n \end{bmatrix}$.  Then 
	$$\begin{bmatrix} k^2 \vert v^*\vert_{n,m} & v \\ v^* & 0\end{bmatrix} = \gamma^* \begin{bmatrix} \vert v^*\vert_{n,m} & v \\ v^* & 0\end{bmatrix} \gamma \in M_{m+n}(V)^+.$$ 
	It follows that $\begin{bmatrix} k \vert v^*\vert_{n,m} & \pm v \\ \pm v^* & 0\end{bmatrix} \in M_{m+n}(V)^+$ for all $k\in \mathbb{R},~k> 0.$ Since  $M_{m+n}(V)^+$ is Archimedean, we have $\pm \begin{bmatrix} 0 & v \\ v^* & 0\end{bmatrix} \in M_{m+n}(V)^+$. Now as $M_{m+n}(V)^+$ is proper, we conclude that $v=0.$ 
	\end{proof} 
\begin{remark} 
	Let $V$ be an absolute order unit space and let $v \in M_{m,n}(V)$ for some $m,n \in \mathbb{N}.$ Then $v=0$ if and only if $\vert v\vert_{m,n}=0$.
\end{remark}

\begin{proposition}
	Let $V$ be an absolute matrix order unit space and let $v \in V.$ Then $\Vert \vert v \vert \Vert = \Vert v \Vert = \Vert \vert v^* \vert \Vert.$
\end{proposition}

\begin{proof}
	If $v=0,$ then $\vert v \vert = 0$ and $\vert v^* \vert = 0$ so that the result holds trivially. Thus we may assume that $v \neq 0$. Then $\vert v \vert \neq 0$ and $\vert v^* \vert \neq 0$. As
	\begin{eqnarray*}
		\Vert v\Vert &=& \left \Vert \begin{bmatrix} 0 & v \\ v^* & 0 \end{bmatrix}\right \Vert_2  
		= \left \Vert \left \vert \begin{bmatrix} 0 & v \\ v^* & 0 \end{bmatrix}\right \vert_2 \right \Vert_2 
		= \left \Vert \begin{bmatrix} \vert v^*\vert & 0 \\ 0 & \vert v\vert \end{bmatrix} \right \Vert_2 \\ 
		&=& max \lbrace \Vert \vert v\vert \Vert, \Vert \vert v^*\vert \Vert \rbrace,
	\end{eqnarray*}
	we have $\Vert \vert v\vert \Vert, \Vert \vert v^*\vert \Vert \leq \Vert v\Vert.$ Now, as in the proof of Proposition \ref{6}, we may deduce that $\begin{bmatrix} k \vert v^*\vert & \pm v \\ \pm v^* & k^{-1} \vert v\vert \end{bmatrix} \in M_2(V)^+$ for all $k \in \mathbb{R}, k > 0.$ Put $k = \sqrt{\frac{\Vert \vert v\vert \Vert}{\Vert \vert v^*\vert \Vert}}$. Since $\vert v\vert \leq \Vert \vert v \vert \Vert e$ and since $\vert v^*\vert \leq \Vert \vert v^* \vert \Vert e$, we may further conclude that $\begin{bmatrix} l e & \pm v \\ \pm v^* & l e \end{bmatrix} \in M_2(V)^+$ where $l = \sqrt{\Vert \vert v\vert \Vert \Vert \vert v^*\vert \Vert}$. Thus $\Vert v\Vert \leq \sqrt{\Vert \vert v\vert \Vert \Vert \vert v^*\vert \Vert}$. Now 
	$$\sqrt{\Vert \vert v\vert \Vert \Vert \vert v^*\vert \Vert} \leq max \lbrace \Vert \vert v\vert \Vert,\Vert \vert v^*\vert \Vert \rbrace = \Vert v\Vert$$ 
	so that $\Vert v\Vert = \sqrt{\Vert \vert v\vert \Vert \Vert \vert v^*\vert \Vert}.$ Hence $\Vert \vert v \vert \Vert = \Vert v \Vert = \Vert \vert v^* \vert \Vert$.
\end{proof}

\section{Partial isometry}
C$^*$-algebras are the primary source of examples for absolute matrix order unit spaces. One of the aims of this paper is to generalize and study certain C$^*$-algebraic notions in the order theoretic context. 
\begin{definition}
Let $(V,e)$ be an absolute matrix order unit space and let $v \in M_n(V)$ for some $n\in \mathbb{N}.$ 
\begin{enumerate}
	\item $v$ is said to be \emph{normal}, if $\vert v \vert_n = \vert v^* \vert_n$; 
	\item $v$ is said to be a \emph{unitary}, if $\vert v \vert_n = \vert v^* \vert_n = e^n$; 
	\item $v$ is said to be a \emph{symmetry}, if $v^* = v$ and $\vert v \vert_n = e^n$;
	\item $v$ is said to be an \emph{order projection}, if $v^* = v$ and $\vert 2 v - e^n \vert_n = e^n$; 
	\item $v$ is said to be a \emph{partial unitary}, if $\vert v \vert_n = \vert v^* \vert_n$ is an order projection; 
	\item $v$ is said to be a \emph{partial symmetry}, if $v^* = v$ and $\vert v \vert_n$ is an order projection; 
	
	Now let $v \in M_{m,n}(V)$ for some $m, n \in \mathbb{N}$ .
	\item $v$ is said to be a \emph{partial isometry}, if $\vert v \vert_{m,n}$ and $\vert v^* \vert_{n,m}$ are order projections;
	\item $v$ is said to be an \emph{isometry}, if $v$ is a partial isometry and $\vert v \vert_{m,n} = e^n$; 
	\item $v$ is said to be a \emph{co-isometry}, if $v$ is a partial isometry and $\vert v^* \vert_{n,m} = e^m$.
\end{enumerate} 
\end{definition}
In this paper, we shall keep our focus on partial isometries, vis-a-vis on order projections. The set of all partial symmetries in $M_{m,n}(V)$ will be denoted by $\mathcal{PI}_{m,n}(V)$ and the set of all order projections in $M_n(V)$ will be denoted by $\mathcal{OP}_n(V)$. For $m = n$, we write $\mathcal{PI}_{m,n}(V) = \mathcal{PI}_n(V)$. For $n = 1$, we shall write $\mathcal{PI}(V)$ for $\mathcal{PI}_1(V)$ and $\mathcal{OP}(V)$ for $\mathcal{OP}_1(V)$. 

The first author introduced and studied order projections defined in an absolute order unit space \cite{K18}. As partial isometries are defined in terms of order projections, the following properties of order projections will be helpful in the study of partial isometries and order projections in the matricial context. 
\begin{remark}\label{7} \cite{K18} 
	Let $(V,V^+,|\cdot|,e)$ be an absolute order unit space. 
	\begin{enumerate}
		\item Then $\mathcal{OP}(V)=\lbrace p\in [0,e]:p\perp (e-p)\rbrace.$ 
		\item For $p,q \in \mathcal{OP}(V),$ the following statements are equivalent:
		\begin{enumerate}
			\item $p+q\leq e;$
			\item $p\perp q;$
			\item $p+q \in \mathcal{OP}(V),$ and
			\item $p\perp_\infty q.$
		\end{enumerate} 
		\item Let $0\leq u,v \leq e.$ If $u+v \in \mathcal{OP}(V)$ with $u \perp v,$ then $u,v \in \mathcal{OP}(V).$
	\end{enumerate} 
	These statements remain valid in an absolute matrix order unit space $V$, if we replace $\mathcal{OP}(V)$ by $\mathcal{OP}_n(V)$ for any $n \in \mathbb{N}$. 
\end{remark} 
\begin{remark}\label{8}
	Let $V$ be an absolute matrix order unit space and let $m,n \in \mathbb{N}$. 
	\begin{enumerate}
		\item It follows from Proposition \ref{5}(2) that if $u,v \in \mathcal{PI}_{m,n}(V)$ with $u \perp v,$ then $\begin{bmatrix}u \\ v\end{bmatrix} \in \mathcal{PI}_{2m,n}(V)$ and $\begin{bmatrix}u & v\end{bmatrix} \in \mathcal{PI}_{m,2n}(V).$
		\item Let $v \in \mathcal{PI}_{m,n}(V)$. Then $\begin{bmatrix} v & 0\end{bmatrix} \in \mathcal{PI}_m(V)$, if $m > n$ and $\begin{bmatrix} v \\ 0\end{bmatrix} \in \mathcal{PI}_n(V)$, if $m < n$.
		\item Let $v \in M_{m,n}(V).$ Then the following facts are equivalent:
		\begin{enumerate}
			\item $v \in \mathcal{PI}_{m,n}(V)$;
			\item $\begin{bmatrix} v^* & 0 \\ 0 & v \end{bmatrix}\in \mathcal{PI}_{m+n}(V)$; and 
			\item $\begin{bmatrix}0 & v \\ v^* & 0\end{bmatrix} \in \mathcal{PI}_{m+n}(V)$. (That is, $\begin{bmatrix}0 & v \\ v^* & 0\end{bmatrix}$ is a partial symmetry.)
		\end{enumerate}
		In fact, we have 
		$$\left \vert \begin{bmatrix} 0 & v \\ v^* & 0\end{bmatrix}\right \vert_{m+n} = \begin{bmatrix} \vert v^*\vert_{n,m} & 0 \\ 0 & \vert v\vert_{m,n}\end{bmatrix} = \left \vert\begin{bmatrix}v^* & 0 \\ 0 & v\end{bmatrix}\right \vert_{m+n}.$$ 
	\end{enumerate}
\end{remark}

\begin{proposition} \label{9}
	Let $V$ be an absolute matrix order unit space and let $v\in M_n(V)_{sa}$ for some $n\in \mathbb{N}.$ Then $v\in \mathcal{PI}_n(V)$ (that is, $v$ is a partial symmetry) if, and only if, $v^\pm \in \mathcal{OP}_{n}(V).$ In particular, if $v \in M_{m,n}(V)$ for some $m,n \in \mathbb{N},$ then $v \in \mathcal{PI}_{m,n}(V)$ if, and only if, $\begin{bmatrix} 0 & v \\ v^* & 0\end{bmatrix}^\pm \in \mathcal{OP}_{m+n}(V).$ 
\end{proposition}

\begin{proof}
	Since $v \in M_n(V)_{sa},$ let $v = v^+ - v^-$ be orthogonal decomposition of $v$ in $M_n(V)^+$. Thus $\vert v\vert = v^+ + v^-$. Now first suppose that $v \in \mathcal{PI}_n(V).$ Then  $\vert v\vert \in \mathcal{OP}_n(V),$ therefore by Remark \ref{7}(3), we get that $v^\pm \in \mathcal{OP}_n(V)$ as $v^+ \perp v^-$.
	
	Conversely suppose that $v^\pm \in \mathcal{OP}_n(V).$ Again by Remark \ref{7}(2), we get that $\vert v\vert = v^+ + v^- \in \mathcal{OP}_n(V)$ so that $v \in \mathcal{PI}_n(V).$
	
	Now, let $v \in M_{m,n}(V)$ for some $m,n \in \mathbb{N}$. Then $\begin{bmatrix}0 & v \\ v^* & 0\end{bmatrix} \in M_{m+n}(V)_{sa}$. By Remark \ref{8}(3), we have that $v \in \mathcal{PI}_{m,n}(V)$ if, and only if, $\begin{bmatrix}0 & v \\ v^* & 0\end{bmatrix} \in \mathcal{PI}_{m+n}(V)$. Hence $v \in \mathcal{PI}_{m,n}(V)$ if, and only if, $\begin{bmatrix} 0 & v \\ v^* & 0\end{bmatrix}^\pm \in \mathcal{OP}_{m+n}(V).$  
\end{proof}
\begin{corollary}\label{10}
	Let $V$ be an absolute matrix order unit space and let $v_1, \dots, v_k \in M_{m,n}(V)$ be mutually orthogonal vectors for some $k,m,n \in \mathbb{N}$. Then $v_1, \dots, v_k \in \mathcal{PI}_{m,n}(V)$ if, and only if, $\displaystyle\sum_{i=1}^{k}v_i \in \mathcal{PI}_{m,n}(V).$ 
\end{corollary}

\begin{proof}
	Put $w_i = \begin{bmatrix} 0 & v_i \\ v_i^* & 0\end{bmatrix}$ for all $i=1,2,\cdots,k$. By Proposition \ref{1}, we get that $\left \lbrace \begin{bmatrix} 0 & v_i \\ v_i^* & 0\end{bmatrix}^{\pm}: i = 1,2, \cdots, k \right \rbrace$ is an orthogonal set so that $\left(\mathlarger{\mathlarger{\mathlarger{\sum}}}_{i=1}^{k} \begin{bmatrix} 0 & v_i \\ v_i^* & 0 \end{bmatrix}\right)^+ = \mathlarger{\mathlarger{\mathlarger{\sum}}}_{i=1}^{k} \begin{bmatrix} 0 & v_i \\ v_i^* & 0 \end{bmatrix}^+$ and $\left(\mathlarger{\mathlarger{\mathlarger{\sum}}}_{i=1}^{k} \begin{bmatrix} 0 & v_i \\ v_i^* & 0 \end{bmatrix}\right)^- = \mathlarger{\mathlarger{\mathlarger{\sum}}}_{i=1}^{k} \begin{bmatrix} 0 & v_i \\ v_i^* & 0 \end{bmatrix}^-$. Then by Remarks \ref{7}, \ref{8}(3) and by Proposition \ref{9}, it follows that $v_1, \dots, v_k \in \mathcal{PI}_{m,n}(V)$ if, and only if, $\displaystyle\sum_{i=1}^{k}v_i \in \mathcal{PI}_{m,n}(V).$ 
\end{proof}
Now, we prove some matricial properties of order projections.
\begin{proposition}\label{11}
	Let $V$ be an absolute matrix order unit space and let $m,n \in \mathbb{N}$. Then $p\in \mathcal{OP}_m(V),q\in \mathcal{OP}_n(V)$ if, and only if, $p \oplus q \in \mathcal{OP}_{m+n}(V).$
\end{proposition} 
\begin{proof}
	Let $p\oplus q\in OP_{m+n}(V)$. Then 
	$$e^{m+n} = \vert 2(p \oplus q) - e^{m+n} \vert_{m+n} = \vert 2 p - e^m \vert_m \oplus \vert 2 q - e^n \vert_n.$$ 
	Thus $\vert 2 p - e^m \vert_m = e^m$ and $\vert 2 q - e^n \vert_n = e^n$ so that $p\in \mathcal{OP}_m(V)$ and $q\in \mathcal{OP}_n(V)$. Now tracing back the proof, we can prove the converse part. 
\end{proof}

\begin{lemma}\label{12}
	Let $V$ be an absolutely matrix ordered space and let $n\in \mathbb{N}.$ Then $\vert \alpha^*v\alpha\vert_n=\alpha^*\vert v\vert_n\alpha$, if $v\in M_n(V)$ and $\alpha \in M_n$ is a unitary.
\end{lemma}

\begin{proof}
	By \ref{2}(2), we have 
	$$\vert \alpha^* v \alpha \vert_n \leq \Vert \alpha^* \Vert \vert \vert v \vert_n \alpha \vert_n = \vert \vert v  \vert_n \alpha \vert_n = \vert \alpha (\alpha^* \vert v \vert_n \alpha) \vert_n \leq \alpha^* \vert v \vert_n \alpha.$$
	Thus $\vert \alpha^*v\alpha \vert_n\leq \alpha^*\vert v\vert_n \alpha.$ Now replacing $v$ by $\alpha v\alpha^*,$ we get 
	$\vert v\vert_n\leq \alpha^*\vert \alpha v\alpha^*\vert_n\alpha$ so that $\alpha\vert v\vert_n\alpha^*\leq\vert\alpha v\alpha^*\vert_n.$
	Finally, interchanging $\alpha$ and $\alpha^*$, we get that $\alpha^*\vert v\vert_n \alpha\leq\vert \alpha^*v\alpha \vert_n.$ Hence $\vert \alpha^*v\alpha \vert_n= \alpha^*\vert v\vert_n \alpha.$ 
\end{proof} 

\begin{proposition}
	Let $(V, e)$ be an absolute matrix order unit space and let $\alpha \in M_{m,n}$ such that $\alpha^* \alpha = I_n$. Then $\alpha u \alpha^* \perp \alpha v \alpha^*$ for $u, v \in M_n(V)^+$ with $u \perp v.$ In particular, if $p \in \mathcal{OP}_n(V),$ then $\alpha p \alpha^* \in \mathcal{OP}_m(V)$.
\end{proposition}

\begin{proof}
	Let $u,v \in M_n(V)^+$ such that $u \perp v.$ Since $\alpha^*\alpha=I_n,$ we have that rank$(\alpha)=n$ so that $ n\leq m.$ Thus we can find an unitary $\delta\in M_m$ such that $\alpha=\delta\begin{bmatrix} I_n\\ o_{m-n,n}\end{bmatrix}.$ Then 
	$$\alpha u\alpha^*=\delta\begin{bmatrix} I_n\\ o_{m-n,n} \end{bmatrix}u\begin{bmatrix} I_n & o_{n,m-n}\end{bmatrix}\delta^*=\delta (u\oplus 0_{m-n})\delta^*.$$ 
	Similarly $\alpha v\alpha^* = \delta (v\oplus 0_{m-n})\delta^*.$ Now by Lemma \ref{12}, we get 	
	\begin{eqnarray*}
		\vert \alpha u\alpha^*-\alpha v\alpha^*\vert_m &=& \vert \delta (u\oplus 0_{m-n})\delta^*-\delta (v\oplus 0_{m-n})\delta^*\vert_m\\
		&=& \delta \vert (u-v)\oplus 0_{m-n}\vert_m \delta^* \\
		&=& \delta (\vert u-v\vert_n\oplus 0_{m-n})\delta^* \\
		&=& \delta ((u+v)\oplus 0_{m-n})\delta^* \\
		&=&\alpha u\alpha^* + \alpha v\alpha^*
	\end{eqnarray*}
	so that $\alpha u\alpha^*\perp \alpha v\alpha^*.$
	
	Next let $p\in \mathcal{OP}_n(V)$ and put $r=p\oplus 0_{m-n}.$ Then $r \in \mathcal{OP}_m(V)$ and $\alpha p\alpha^*=\delta r\delta^*.$ Thus $r \perp (e^m - r)$ so that, by the first part, we have $\delta r\delta^* \perp \delta (e^m-r)\delta^*=e^m-\delta r\delta^*$. Hence $\alpha p\alpha^* \in \mathcal{OP}_m(V).$
\end{proof}

\section{Comparison of order projections} 
Recall that a pair of projections $p$ and $q$ in a C$^*$-algebra $A$ is said to (Murray-von Neumann) equivalent if there exists a partial isometry $u \in A$ such that $p = u^* u$ and $q = u u^*$, or equivalently, $p = \vert u \vert$ and $q = \vert u^* \vert$. The alternative form provides room to extend the notion to absolute matrix order unit spaces. However, before we formally introduce the notion, let us look at the other aspects. 

First, let us note that existence of a partial isometry is identical to that of a (Murray-von Neumann) equivalent pair of projections. In other words, the study of equivalence of a pair of projections is the study of partial isometries. Being multiplicative in nature, a partial isometry enjoys certain properties with order theoretic connotation which we have not been able to prove order theoretically.
\begin{enumerate}
	\item Let $u$ be a partial isometry in a C$^*$-algebra $A$ so that $\vert u \vert$ and $\vert u^* \vert$ are projections and let $p$ be projection in $A$. Then $p \le \vert u \vert$ if and only if there exists a partial isometry $u_1 \in A$ with $u_1 \perp (u - u_1)$ such that $p = \vert u_1 \vert$. In fact, if $p \le \vert u \vert$, we put $u_1 = u p$. Then $u_1$ is also a partial isometry in $A$ with $u_1 \perp (u - u_1)$ such that $p = \vert u_1 \vert$. Conversely, if $u_1$ is a partial isometry in $A$ with $u_1 \perp (u - u_1)$ such that $p = \vert u_1 \vert$, then $p \le \vert u \vert$. 
	\item Let $u$ and $v$ be any two partial isometries in $A$ such that $\vert u \vert = \vert v \vert$. Put $w = v u^*$. Then $w^* w = u u^*$ and $w w^* = v v^*$, that is, $\vert w \vert = \vert u^* \vert$ and $\vert w^* \vert = \vert v^* \vert$. 
\end{enumerate} 
These properties can be extended to $M_{m,n}(A)$ for any $m, n \in \mathbb{N}$. We propose them in terms of the following conditions on an absolute matrix order unit space $V$:  
\begin{enumerate}
	\item[(H)] If $u \in \mathcal{PI}_{m,n}(V)$ and if $p \in \mathcal{OP}_n(V)$ with $p \le \vert u \vert_{m,n}$, then there exists a $u_1 \in \mathcal{PI}_{m,n}(V)$ with $u_1 \perp (u - u_1)$ such that $p = \vert u_1 \vert$. 
	\item[(T)] If $u \in \mathcal{PI}_{m,n}(V)$ and $v \in \mathcal{PI}_{l,n}(V)$ with $\vert u \vert_{m,n} = \vert v \vert_{l,n}$, then there exists a $w \in \mathcal{PI}_{m,l}(V)$ such that $\vert w^* \vert_{l,m} = \vert u^* \vert_{n,m}$ and $\vert w \vert_{m,l} = \vert v^* \vert_{n,l}$.
\end{enumerate} 

\begin{definition}
Let $V$ be an absolute matrix order unit space. We write 
$$\mathcal{OP}_{\infty}(V)= \lbrace p: p \in \mathcal{OP}_n(V) ~ \textrm{for some} ~ n \in \mathbb{N} \rbrace.$$
Given $p\in \mathcal{OP}_m(V)$ and $q\in \mathcal{OP}_n(V)$, we say that $p$ is \emph{partial isometric equivalent} to $q$, (we write, $p\sim q$), if there exists a $v\in \mathcal{PI}_{m,n}(V)$ such that $p=\vert v^*\vert_{n,m}$ and $q=\vert v\vert_{m,n}$. 
\end{definition}
\begin{lemma}\label{13} 
	Let $V$ be an absolute matrix order unit space satisfying condition (H) and let $p \in \mathcal{OP}_m(V)$ and $q \in \mathcal{OP}_n(V)$ for some $m, n \in \mathbb{N}$ with $p \sim q$. If $p_1 \leq p$ for some $p_1 \in \mathcal{OP}_m(V)$, then there exists $q_1 \in \mathcal{OP}_n(V)$ with $q_1 \le q$ such that $p_1 \sim q_1$ and $(p - p_1) \sim (q - q_1)$. 
\end{lemma}
\begin{proof} 
	Since $p \sim q$, there exists $u \in \mathcal{PI}_{m,n}(V)$ such that $p=\vert u^* \vert_{n,m}$ and $q=\vert u \vert_{m,n}$. As $p_1 \le p = \vert u^* \vert_{n,m}$, by condition (H), there exists $v \in \mathcal{PI}_{m,n}(V)$ with $v \perp (u - v)$ such that $\vert v^* \vert_{n,m} = p_1$. Put $q_1 = \vert v \vert_{m,n}$. Then $p_1 \sim q_1$ and $q = \vert u \vert_{m,n} = \vert v \vert_{m,n} + \vert (u - v) \vert_{m,n} \geq q_1$. Further, $\vert (u - v)^* \vert_{n,m} = p - p_1$ and $\vert (u - v) \vert_{m,n} = q - q_1$ so that $(p - p_1) \sim (q - q_1)$. 
\end{proof}
\begin{remark} 
	Let $V$ be an absolute matrix order unit space satisfying condition (H) and assume that $p_i, p \in \mathcal{OP}_m(V)$, $i = 1, \dots k$ and $q \in \mathcal{OP}_n(V)$ for some $k, m, n \in \mathbb{N}$. If $p_i \perp p_j$ for $i \neq j$ and if $p_i \leq p \sim q$ for each $i = 1, \dots k$, then there exists $q_i \in \mathcal{OP}_n(V)$, $i = 1, \dots k$ such that $q_i \leq q$, $p_i \sim q_i$ for all $i = 1, \dots k$ and $q_i \perp q_j$ whenever $i \neq j$. 
\end{remark}
\begin{proposition}\label{l4}
Let $V$ be an absolute matrix order unit space satisfying the condition (T). Then $\sim$ is an equivalence relation on $\mathcal{OP}_\infty(V).$
\end{proposition} 
\begin{proof}
	(a) Let $p \in \mathcal{OP}_{\infty}(V)$. Considering $p \in \mathcal{PI}_{\infty}(V)$, we get $p \sim p$.
	
	(b) Let $p\in \mathcal{OP}_m(V)$ and $q\in \mathcal{OP}_n(V)$ be such that $p\sim q$. Then there exists a $v\in \mathcal{PI}_{m,n}(V)$ such that $p=\vert v^*\vert_{n,m}$ and $q=\vert v\vert_{n,m}$. Now, for $w = v^* \in \mathcal{PI}_{n,m}(V)$, we have $q = \vert w^* \vert_{m,n}$ and $p = \vert w \vert_{n,m}$. Thus $q \sim p$. 
	
	(c) Let $p\in \mathcal{OP}_k(V)$, $q\in \mathcal{OP}_l(V)$ and $r\in \mathcal{OP}_m(V)$ be such that $p \sim q$ and $q \sim r$. Then there exists $u \in \mathcal{PI}_{k,l}(V)$ and $v \in \mathcal{PI}_{l,m}(V)$ such that $p = \vert u^* \vert_{l,k}$, $q = \vert u \vert_{k,l}$, $q = \vert v^*\vert_{m,l}$ and $r = \vert v \vert_{l,m}$. As $\vert u \vert_{k,l} = \vert v^*\vert_{m,l}$, by condition (T), there exists $w \in \mathcal{PI}_{k,m}(V)$ such that $\vert w \vert_{k,m} = \vert v \vert_{l,m} = r$ and $\vert w^* \vert_{m,k} = \vert u^* \vert_{l,k} = p$. Thus $p \sim r$.
\end{proof}
\begin{proposition}\label{15}
	Let $V$ be an absolute matrix order unit space and let $p,q,r,p',q'\in \mathcal{OP}_\infty(V)$. Then 
	\begin{enumerate}
		\item If $m, n \in \mathbb{N}$ and let $p\in\mathcal{OP}_m(V)$, then $p\sim p\oplus 0_n$ and $p \sim 0_n \oplus p$; 
		\item If $p \sim q$ and $p' \sim q'$ with $p \perp p'$ and $q \perp q'$, then $p + p' \sim q + q'$;
		\item If $p\sim p'$ and $q\sim q',$ then $p\oplus q\sim p'\oplus q';$
		\item $p\oplus q\sim q\oplus p;$
		\item If $p,q\in \mathcal{OP}_n(V)$ for some $n\in \mathbb{N}$ such that $p\perp q,$ then $p+q\sim p\oplus q;$
		\item $(p\oplus q)\oplus r=p\oplus (q\oplus r).$
	\end{enumerate} 
\end{proposition}
\begin{proof} 
	\begin{enumerate}
		\item Let $m, n \in \mathbb{N}$ and let $p \in \mathcal{OP}_m(V)$. Put $v = \begin{bmatrix} p \\ 0_{n,m} \end{bmatrix}$. Then by Remark \ref{3}(4) and (5), we have $\vert v \vert = p$ and $\vert v^* \vert = p \oplus 0_n$. Thus $p \sim p \oplus 0_n$. Similarly, $w = \begin{bmatrix} 0_{n,m} \\ p \end{bmatrix}$ yields $\vert w \vert = p$ and $\vert w^* \vert = 0_n \oplus p$ so that $p \sim 0_n \oplus p$. 
		\item For definiteness, let $p, p' \in \mathcal{OP}_m(V)$ and $q, q' \in \mathcal{OP}_n(V)$. Since $p \perp p'$ and $q \perp q'$, by Remark \ref{7}(2), we have $p + p' \in \mathcal{OP}_m(V)$ and $q + q' \in \mathcal{OP}_n(V)$. Since $p \sim q$ and $p' \sim q'$, there exist $v, v' \in \mathcal{PI}_{m,n}(V)$ such that $p = \vert v^* \vert_{n,m}$, $q = \vert v \vert_{m,n}$, $p' = \vert v'^*\vert_{n,m}$ and $q' = \vert v' \vert_{m,n}$. As $p \perp p'$ and $q \perp q'$, by Remark \ref{4}, we get that $v \perp v'$. Now by Corollary \ref{10}, we may conclude that $v + v' \in \mathcal{PI}_{m,n}$. Also then 
		$$\vert v + v' \vert_{m,n} = \vert v \vert_{m,n} + \vert v' \vert_{m,n} = q + q'$$ 
		and 
		$$\vert v^* + v'^* \vert_{n,m} = \vert v^* \vert_{n,m} + \vert v'^* \vert_{n,m} = p + p'$$ 
		so that $p + p' \sim q + q'$. 
		\item Since $p \sim q$ and $p' \sim q'$, there exist $v \in \mathcal{PI}_{m,n}(V)$ and $v' \in \mathcal{PI}_{r,s}(V)$ such that $p = \vert v^* \vert_{n,m}$, $q = \vert v \vert_{m,n}$, $p' = \vert v'^*\vert_{s,r}$ and $q' = \vert v' \vert_{r,s}$. Put $w = \begin{bmatrix} v & 0 \\ 0 & v'\end{bmatrix}$. Then $\vert w^*\vert_{n+s,m+r} = p \oplus q$ and $\vert w\vert_{m+r,n+s} = p' \oplus q'$. Hence $p \oplus q \sim p' \oplus q'$.
		
		\item Put $v=\begin{bmatrix} 0 & q\\ p & 0 \end{bmatrix}.$ Then $\begin{bmatrix}0 & I_m\\ I_n & 0\end{bmatrix}\begin{bmatrix} 0 & q\\ p & 0\end{bmatrix}=p\oplus q$ and $\begin{bmatrix}0 & I_n\\ I_m & 0\end{bmatrix}\begin{bmatrix} 0 & q \\ p & 0\end{bmatrix}^*=q\oplus p$, where $\begin{bmatrix}0 & I_m\\ I_n & 0\end{bmatrix}$ and $\begin{bmatrix}0 & I_n\\ I_m & 0\end{bmatrix}$ are isometries in $M_{m+n}.$ Therefore by Remark \ref{3}(1), we get that $p\oplus q=\left\vert \begin{bmatrix} 0 & q\\ p & 0 \end{bmatrix}\right\vert$ and $q\oplus p=\left\vert\begin{bmatrix} 0 & q\\ p & 0\end{bmatrix}^*\right\vert.$ Hence $ p \oplus q\sim q \oplus p.$
		
		\item By (1), we get that $p \sim p \oplus 0_n$ and $q \sim 0_n \oplus q$. Also $p \oplus 0_n \perp 0_n \oplus q$. Thus by (2), we may conclude that $p + q \sim p \oplus q$.
	\end{enumerate} 
	Now, (6) can be proved in the same way.
\end{proof} 
\begin{definition}
	Let $V$ be an absolute matrix order unit space and let $p, q \in \mathcal{OP}_\infty(V)$. We say that $p\preceq q$, if there exists $r \in \mathcal{OP}_\infty(V)$ such that $p \sim r \leq q$. 
\end{definition}

\begin{proposition}\label{16}
	Let $V$ be an absolute matrix order unit space satisfying conditions (T) and (H). 
	\begin{enumerate}
		\item Then $\preceq$ is a reflexive and transitive relation.
		\item Let $p, q, r, s \in \mathcal{OP}_\infty(V)$. If $p \preceq r$ and $q \preceq s$, then $p \oplus q \preceq r \oplus s$.
	\end{enumerate}
\end{proposition}
\begin{proof}
	(1). Let $p, q, r \in \mathcal{OP}_\infty(V)$. Then $p \preceq p$ holds trivially. Next let $p \preceq q$ and $q \preceq r$. Then there exist $p_1, q_1 \in \mathcal{OP}_\infty(V)$ such that $p \sim p_1 \le q$ and $q \sim q_1 \le r$. Since $p_1 \le q \sim q_1$, by Lemma \ref{13}, there exists $q_2 \le q_1$ such that $p_1 \sim q_2$. Since $p \sim p_1$ and since $p_1 \sim q_2$, by Proposition \ref{l4}, we get $p \sim q_2 \le q_1 \le r$. Thus $p \preceq r$. 
	
	(2). Let $p \preceq r$ and $q \preceq s$. Then there are $p_1, q_1 \in \mathcal{OP}_\infty(V)$  such that $p \sim p_1 \le r$ and $q \sim q_1 \le s$. Thus $p \oplus q \sim p_1 \oplus q_1 \le r \oplus s$ so that $p \oplus q \preceq r \oplus s$. 
\end{proof}

\begin{proposition}\label{17}
	Let $V$ be absolute matrix order unit space satisfying conditions (T) and (H) and let $p, q \in \mathcal{OP}_\infty(V)$. Then $p\preceq q$ if and only if $q \sim p \oplus p_0$ for some $p_0 \in \mathcal{OP}_\infty(V)$. 
\end{proposition}
\begin{proof}
	First assume that $p \preceq q$. Then $p \sim r \leq q$ for some $r \in \mathcal{OP}_\infty(V)$. Put $p_0 = (q - r)$. Then $p_0 \in \mathcal{OP}_\infty(V)$ with $p_0 \perp r$. Since $r \sim p$, by Proposition \ref{15}, we get that $q = r + p_0 \sim p \oplus p_0.$
	
	Conversely assume that $q \sim p \oplus p_0$ for some $p_0 \in \mathcal{OP}_\infty(V)$. Then $p \oplus 0 \leq p \oplus p_0 \sim q$. Thus by Lemma \ref{13}, there exist $r \in \mathcal{OP}_\infty(V)$ such that $p \oplus 0 \sim r \leq q$. Now, by Proposition \ref{15}(1), we have $p \sim p \oplus 0$ so that $p \preceq q$.
\end{proof}

\begin{corollary}
	Let $V$ be an absolute matrix order unit space satisfying (T) and let $p_1,p_2,q_1,q_2 \in \mathcal{OP}_\infty(V)$ such that $p_1 \perp p_2$ and $q_1 \perp q_2$. If $p_1 \preceq q_1$ and $p_2 \preceq q_2,$ then $p_1 + p_2 \preceq q_1 + q_2$.
\end{corollary}

\begin{proof}
	Suppose that $p_1 \preceq q_1$ and $p_2 \preceq q_2$. By Proposition \ref{16}(2), we get that $p_1 \oplus  p_2 \preceq q_1 \oplus q_2$. Now $p_1 + p_2 \sim p_1 \oplus p_2$ and $q_1 + q_2 \sim q_1 \oplus q_2$ so that $p_1 + p_2 \preceq p_1 \oplus p_2$ and $q_1 \oplus q_2 \preceq q_1 + q_2$. Therefore, again by Proposition \ref{16}(1), we may conclude that $p_1 + p_2 \preceq q_1 + q_2$.
\end{proof}

\subsection{Unitary equivalence}
\begin{definition}
	Let $V$ be an absolutely matrix ordered space and let $u,v,w \in M_n(V)$ for some $n \in \mathbb{N}$ such that $u=v+w.$ We say that $v$ and $w$ are ortho-component of $u$ if $v \perp w.$ 
\end{definition}
\begin{remark} 
	\begin{enumerate}
		\item Let $V$ be an absolute matrix order unit space. Then a partial isometry in $V$ may be ortho-component of more than one unitary in $V.$ To see this, let $V= M_2(\mathbb{C})$ and put $v = \begin{bmatrix} 0 & 1 \\ 0 & 0\end{bmatrix}$ and $w_c = \begin{bmatrix} 0 & 0 \\ c & 0\end{bmatrix}$ with $\vert c\vert = 1$. Then $u_c := v + w_c$ is a unitary in $V$ and $w_c$ is an ortho-complement of $v$ for any $c \in \mathbb{C}$ with $\vert c\vert = 1$. 
		\item Let $V$ be an absolute matrix order unit space. Then every partial unitary is an ortho-component of some unitary. In fact, if $v$ is partial unitary in $M_n(V)$ for some $n\in \mathbb{N}$, then $v\perp e^n-\vert v\vert_n$ so that $v + e^n-\vert v\vert_n$ is unitary. 
	\end{enumerate}	
\end{remark} 
\begin{definition}
	Let $V$ be an absolute matrix order unit space and let $p, q \in \mathcal{OP}_n(V)$ for some $n \in \mathbb{N}.$ We say that $p$ is \emph{unitary equivalent} to $q$, (we write, $p \sim_u q$), if there exists a unitary $u \in M_n(V)$ and an ortho-component $v$ of $u$ such that $p = \vert v^* \vert_n$ and $q = \vert v \vert_n.$
\end{definition}
Note that if $v$ is an ortho-complement of a unitary $u$, then $v$ and $u - v$ are partial isometries. In the next result, we characterize unitary equivalence of projections in terms of partial isometry equivalence. 
\begin{proposition}\label{18}
	Let $V$ be an absolute matrix order unit space and let $p,q\in \mathcal{OP}_n(V)$ for some $n\in \mathbb{N}.$ Then $p\sim_u q$ if, and only if, $p\sim q$ and $e^n-p \sim e^n-q.$ 
\end{proposition}
\begin{proof}
	First, let $p\sim_u q$. Then there exists an ortho-complement $v$ of a unitary $u$ in $M_n(V)$ such that $p = \vert v^* \vert_n$ and $q = \vert v \vert_n$. Put $w = u - v$. Then $$\vert v \vert_n + \vert w \vert_n = \vert v + w \vert_n = \vert u \vert_n = e^n$$ 
	and 
	$$\vert v^* \vert_n + \vert w^* \vert_n = \vert v^* + w^* \vert_n = \vert u^* \vert_n = e^n.$$ 
	Thus $\vert w^* \vert_n = e^n - p$ and $\vert w \vert_n = e^n - q$ so that $p \sim q$ and $e^n-p \sim e^n-q$. 
	
	Conversely, assume that $p\sim q$ and $e^n-p \sim e^n-q.$ Then there exist partial isometries $v, w \in M_n(V)$ such that $p = \vert v^* \vert_n$, $q = \vert v \vert_n$, $e^n - p = \vert w^* \vert_n$ and $e^n-q = \vert w \vert_n.$ Thus $\vert v \vert_n \perp \vert w \vert_n$ and $\vert v^* \vert_n \perp \vert w^* \vert_n$. Now, by Remark \ref{4}(2) and Proposition \ref{5}(1), we get that $v \perp w$ with $\vert v + w \vert_n = e^n = \vert v^* + w^* \vert_n$. Thus $u := v + w$ is a unitary so that $p\sim_u q.$
\end{proof}
\begin{corollary} 
	Let $V$ be an absolute matrix order unit space satisfying the condition (T).
	\begin{enumerate}
		\item Then $\sim_u$ is an equivalence relation in $\mathcal{OP}_n(V)$ for all $n \in \mathbb{N}$. 
		\item If $p \in \mathcal{OP}_m(V)$ and $q \in \mathcal{OP}_n(V)$, then $p \oplus q \sim_u q \oplus p$. 
		\item If $p, p' \in \mathcal{OP}_m(V)$ and $q, q' \in \mathcal{OP}_n(V)$ with $p \sim_u p'$ and $q \sim_u q'$, then $p \oplus q \sim_u p' \oplus q'$.
	\end{enumerate}
\end{corollary} 
\begin{proof}
	(1) follows from Propositions \ref{l4} and \ref{18}. 
	
	(2). Let $p \in \mathcal{OP}_m(V)$ and $q \in \mathcal{OP}_n(V)$. It follows from Proposition \ref{15}(4) that $p \oplus q \sim q \oplus p$ and $(e^m - p) \oplus (e^n - q) \sim (e^n - q) \oplus (e^m - p)$. Thus $p \oplus q \sim_u q \oplus p$. 
	
	(3) follows from Propositions \ref{15}(3) and \ref{18}. 
\end{proof} 
\begin{corollary}
	Let $V$ be an absolute matrix order unit space satisfying condition (T) and let $p,q\in \mathcal{OP}_n(V)$ for some $n\in \mathbb{N}.$ Then following statements are equivalent:
	\begin{enumerate}
		\item $p \sim q;$
		\item $p \oplus 0 \sim_u q\oplus 0$;
		\item $0 \oplus p \sim_u 0 \oplus q$. 
	\end{enumerate}
\end{corollary}

\begin{proof}
	First assume that $p \sim q$. Then  $p = \vert v^* \vert_n$ and $q = \vert v \vert_n$ for some partial isometry $v \in M_n(V)$. Put $u = \begin{bmatrix} v^* & e^n-q \\ e^n-p & v \end{bmatrix}$ and consider $u_{11} = \begin{bmatrix} v^* & 0 \\ 0 & 0 \end{bmatrix}$, $u_{12} = \begin{bmatrix} 0 & e^n-q \\ 0 & 0 \end{bmatrix}$, $u_{21} = \begin{bmatrix} 0 & 0 \\ e^n-p & 0 \end{bmatrix}$ and $u_{22} = \begin{bmatrix} 0 & 0 \\ 0 & v \end{bmatrix}$. Then $u = u_{11} + u_{12} + u_{21} + u_{22}$. Also, we have $\vert u_{11} \vert_{2n} = p \oplus 0$, $\vert u_{12} \vert_{2n} =  0 \oplus (e^n - q)$, $\vert u_{21} \vert_{2n} = (e^n - p) \oplus 0$ and $\vert u_{22} \vert_{2n} = 0 \oplus q$. Similarly, $\vert u_{11}^* \vert_{2n} = q \oplus 0$, $\vert u_{12}^* \vert_{2n} = (e^n - q) \oplus 0$, $\vert u_{21}^* \vert_{2n} = 0 \oplus (e^n - p) $ and $\vert u_{22}^* \vert_{2n} = 0 \oplus p$. Thus $u_{11}$, $u_{12}$, $u_{21}$ and $u_{22}$ are mutually orthogonal partial isometries in $M_n(V)$ and subsequently, $u$ is a unitary in $M_n(V)$. In particular, $p\oplus 0_n \sim_u q\oplus 0_n$ and $0_n\oplus p\sim_u 0_n\oplus q$. Therefore, (1) implies (2) and (3). 
	
	Conversely assume that $p \oplus 0\sim_u q \oplus 0$. Then by Proposition \ref{18}, $p \oplus 0 \sim q \oplus 0$. Thus by Propositions \ref{15}(1) and \ref{l4}, we may conclude that $p \sim q$. 
\end{proof}
\begin{proposition}\label{19}
	Let $(V,e)$ be an absolute matrix order unit space and let $p \in \mathcal{OP}_n(V)$ for some $n \in \mathbb{N}$. If $\delta$ is a unitary in $M_n$, then $p\sim_u \delta^* p \delta$. 
\end{proposition}
\begin{proof} 
	Let $\delta \in M_n$ be a unitary. Put $\delta^* p \delta = q$. Consider $v = \delta^* p$. Then $\vert v \vert_n = \vert \delta^* p\vert_n = p$ and $\vert v^* \vert_n = \vert p \delta \vert_n = \vert \delta q \vert_n = q$ for $p\delta = \delta q$. Thus $p \sim q$. Since $e^n - q = e^n - \delta^* p \delta = \delta^* (e^n - p) \delta$, by a similar argument, we can also show that $(e^n - p) \sim (e^n - q)$. Thus $p \sim_u q$. 
\end{proof}
\begin{corollary}
	Let $(V,e)$ be an absolute matrix order unit space satisfying condition (T) and let $p \in \mathcal{OP}_n(V)$ for some $n \in \mathbb{N}$. Then $\alpha p\alpha^* \sim p$ for any $\alpha \in M_{m,n}$ with $\alpha^* \alpha = I_n$. 
\end{corollary}
\begin{proof}
	Let $\alpha \in M_{m,n}$ with $\alpha^* \alpha = I_n$. Find a unitary $\delta \in M_m$ such that $\alpha = \delta \begin{bmatrix} I_n \\ 0_{m-n,n} \end{bmatrix}$. Then $\alpha p\alpha^*=\delta (p \oplus 0_{m-n}) \delta^* \in \mathcal{OP}_m$. Now, by Proposition \ref{19}, we have $\delta (p \oplus 0_{m-n})\delta^* \sim_u (p \oplus 0_{m-n}) \sim p$. Thus $\alpha p \alpha^* \sim p$.
\end{proof}

\section{Infinite projections}
\begin{definition}
Let $V$ be an absolute matrix order unit space and let $p \in \mathcal{OP}_n(V)$ for some $n \in \mathbb{N}$. Then $p$ is said to be \emph{infinite}, if there exists $q \in \mathcal{OP}_n(V)$ with $q \le p$ and $q \neq p$ such that $p \sim q$. We say that $p$ is \emph{finite}, if it is not infinite.
\end{definition} 

\begin{remark}\label{20} 
	Let $V$ be an absolute matrix order unit space.
	\begin{enumerate}
		\item Let $p,q \in \mathcal{OP}_n(V)$ for some $n \in \mathbb{N}$ such that $p \leq q$. If $q$ is finite, then so is $p$. 
		\item Let $p_1, p_2, q \in \mathcal{OP}_n(V)$ with $p_1 \leq q$, $p_2 \leq q$ and $p_1 \perp p_2$. If $q$ is finite, then so is $p_1 + p_2$.   
	\end{enumerate} 
\end{remark}

\begin{corollary}
	Let $V$ be an absolute matrix order unit space and let $n \in \mathbb{N}$. Then the following statements are equivalent:
	\begin{enumerate}
		\item[(1)] $e^n$ is finite;
		\item[(2)] Every isometry in $M_n(V)$ is unitary;
		\item[(3)] $p$ is finite for all $p \in \mathcal{OP}_n(V)$.
	\end{enumerate}
\end{corollary}

\begin{proof}
	The equivalence of (1) and (3) follows from Remark \ref{20}(1).
	
	(1) implies (2): Let $v\in M_n(V)$ is an isometry, then $\vert v^*\vert_n \sim \vert v\vert_n =e^n.$ Since $\vert v^*\vert_n \in \mathcal{OP}(V),$ we have that $\vert v^*\vert_n \leq e^n.$ By assumption, we get that $\vert v^*\vert_n = e^n$ so that $v$ is unitary.
	
	(2) implies (3): Let $p,q\in \mathcal{OP}_n(V)$ such that $p\sim q\leq p.$ There exists $v\in V$ with $\vert v\vert_n = p$ and $\vert v^*\vert_n = q.$ Since $q \leq p,$ we get that $v \perp (e^n-p)$. Put $w= v+(e^n-p).$ Then $\vert w\vert_n = \vert v\vert_n + (e^n-p)=e^n$ and $\vert w^*\vert_n = \vert v^*\vert_n + (e^n-p)=q+(e^n-p).$ By assumption, we have that $\vert w^*\vert_n=e^n$ so that $q=p.$ Hence every projection in $M_n(V)$ is finite.
\end{proof}

In the next result, we characterize infinite projections. For this, we need the following: 
\begin{lemma}\label{21}
	Let $V$ be an absolute matrix order unit space satisfying conditions (T) and (H) and let $p \sim q$ for some $p \in \mathcal{OP}_m(V)$ and $q \in \mathcal{OP}_n(V)$, $m, n \in \mathbb{N}$. If $p$ is infinite, then $q$ is also infinite. 
\end{lemma} 
\begin{proof} 
	Assume that $p$ is infinite. Then $p \sim r \leq p$ with $r \neq p$ for some $r \in \mathcal{OP}_m(V)$.  By Lemma \ref{13}, there exists $s \in \mathcal{OP}_n(V)$ such that $s \leq q$ and $r \sim s$ and $(p - r) \sim (q - s)$. As $p \sim q$, $p \sim r$ and $r \sim s$, by Proposition \ref{l4}, we have $q \sim s$. Since $(p-r) \sim (q-s)$, we also have $0 \neq \Vert p - r \Vert_m = \Vert q - s \Vert_n$. Thus $q \neq s$ so that $q$ is also infinite.
\end{proof} 
\begin{theorem}
	Let $V$ be an absolute matrix order unit space satisfying condition (T) and let $p \in \mathcal{OP}_n(V) \setminus \lbrace 0 \rbrace$ for some $n \in \mathbb{N}$. Then the following statements are equivalent: 
	\begin{enumerate}
		\item $p$ is an infinite projection; 
		\item there exist $q, r \in \mathcal{OP}_n(V) \setminus \lbrace 0 \rbrace$ with $q \perp r$ such that $p = q + r$ and $p \sim q$; 
		\item $p \oplus 0$ is an infinite projection; 
		\item $0 \oplus p$ is an infinite projection;
		\item $\alpha p \alpha^*$ is an infinite projection whenever $\alpha \in M_{m,n}$ with $\alpha^* \alpha = I_n$;
    \end{enumerate}
    
Moreover, if $V$ also satisfy (H), then the above statements are also equivalent to:		
		
		\begin{enumerate}  
		\item[(6)] $p \oplus q \preceq p$ for some $q \in \mathcal{OP}_\infty(V) \setminus \lbrace 0 \rbrace$.
	\end{enumerate}
\end{theorem} 
\begin{proof}
	(1) is equivalent to (2) by the definition.  
	 
	(1) is equivalent to (3): Suppose that $r \leq p\oplus 0$. Then there exist $s \in \mathcal{OP}_\infty(V)$ such that $s\leq p$ and $r = p \oplus 0$. Now equivalence of (1) and (3) follows from Propositions \ref{l4}, \ref{15}(1) and \ref{11}. 
	
	Similarly, we can show that (1) is equivalent to (4) as well.
	
	(3) is equivalent to (5): Let $\alpha \in M_{m,n}$ such that $\alpha^* \alpha = I_n.$ Then $\alpha p \alpha^* = \delta (p \oplus 0_{m-n}) \delta^*$ for some $\delta$ unitary in $M_m$. Since $\delta$ is unitary, we get that $\alpha p \alpha^*$ is finite if, and only if, $p \oplus 0_{m-n}$ is finite. Hence equivalence of (3) and (5) follows.	
	
	Now suppose that $V$ also satisfy (H).
	
	(1) implies (6): Assume that $p$ is infinite. Then there exists $r \in \mathcal{OP}_\infty(V)$ such that $p \sim r \leq p$ and $r \neq p$. Put $q = p - r$. Then $q \in \mathcal{OP}_\infty(V) \setminus \lbrace 0 \rbrace$. Now, by Proposition \ref{15}, we get that $p = r + q \sim p \oplus q$. Thus $p\oplus q\preceq p$.
	
	(6) implies (1): Finally, assume that there exist $q \in \mathcal{OP}_m(V)\setminus \lbrace 0\rbrace$ such that $p\oplus q \preceq p$ for some $m \in \mathbb{N}$. Then by Proposition \ref{17}, we get that $p \sim p\oplus q \oplus q_0$ for some $q_0 \in \mathcal{OP}_l(V)$, $l \in \mathbb{N}$. Thus 
	$$p\oplus q \oplus q_0 \sim p \sim p\oplus 0_{m+l} \leq p\oplus q \oplus q_0.$$ 
	Since $p\oplus 0_{m+l} \neq p\oplus q \oplus q_0,$ we get that $p\oplus q\oplus q_0$ is an infinite projection. Hence by Lemma \ref{21}, $p$ is also an infinite projection. 
\end{proof}

\begin{remark}
Let $V$ be an absolute matrix order unit space satisfying (T) and let $p \in \mathcal{OP}_n(V), q \in \mathcal{OP}_m(V)$ for some $m,n \in \mathbb{N}$. If $p$ or $q$ is infinite, then so is $p \oplus q$. In fact, in this case, either $p \oplus 0_m$ or $0_n \oplus q$ is infinite so that $p \oplus q$ is also infinite.
\end{remark}
\begin{proposition}
Let $V$ be an absolute matrix order unit space satisfying conditions (T) and (H) and assume that $p\in \mathcal{OP}_m(V)$ such that $p$ is infinite. Then there exists a strictly decreasing sequence of infinite projections $\lbrace r_n\rbrace_{n=1}^{\infty}$ with $r_1=p$ such that $r_n \sim p$ as well as $(r_n - r_{n+1}) \sim (r_{n+1} - r_{n+2})$ for all $n\in \mathbb{N}$. 
\end{proposition}

\begin{proof} 
Let $p = r_1$. As $r_1$ is infinite, there exists $r_2 \in \mathcal{OP}_m(V)$ such that $r_2 \le r_1 \sim r_2$. By Lemma \ref{13}, we can find a projection $r_3 \in \mathcal{OP}_m$ with $r_3 \le r_2$ such that $r_2 \sim r_3$ and $(r_1 - r_2) \sim (r_2 - r_3)$. Now the result follows from induction on $n$.
\end{proof}
\begin{remark}
	Put $r_n - r_{n+1} = s_n$. As $r_{n+1} \le r_n$ with $r_n \neq r_{n+1}$, we have $s_n \in \mathcal{OP} \setminus \lbrace 0 \rbrace$ for all $n \in \mathbb{N}$. Further, as $\sum_{n=1}^{k} s_n = r_1 - r_{k+1} \le r_1$ for all $k \in \mathbb{N}$, we get that $s_m \perp s_n$ with $p = r_1 = \sum_{k=1}^{n} s_k + r_{n+1}$ for all $m, n \in \mathbb{N},m \neq n$.
\end{remark} 
\begin{corollary}
	Let $V$ be an absolute matrix order unit space satisfying conditions (T) and (H) and let $q \preceq p$ for some $p, q \in \mathcal{OP}_\infty(V)$. If $p$ is finite, then $q$ is also finite. 
\end{corollary}

\begin{proof}
	Since $q \preceq p$, by Proposition \ref{17}, we get that $p\sim q \oplus q_0$ for some $q_0 \in \mathcal{OP}_n(V)$. Now, assume that $p$ is finite. Then by Lemma \ref{21}, $q \oplus q_0$ is also finite. Since $q\oplus 0_n \leq q \oplus q_0$, by Remark \ref{20}(1) and Lemma \ref{21}, we may conclude that $q$ is also finite. 
\end{proof}

\begin{proposition}
	Let $V$ be an absolute matrix order unit space satisfying conditions (T) and (H) and let $p \preceq q$ and $q \preceq p$ for some $p, q \in \mathcal{OP}_\infty(V)$. If one of them is finite, then $p \sim q$. 
\end{proposition}
\begin{proof}
	Without any loss of generality, we may assume that $p$ is a finite order projection. As $p \preceq q$ and $q \preceq p,$ by Proposition \ref{17}, we have $q \sim p \oplus p_0$ and $p \sim q\oplus q_0$ for some $p_0, q_0 \in \mathcal{OP}_\infty(V)$. Thus by Proposition \ref{15}, we get that $p \sim (p \oplus p_0) \oplus q_0$.  Then we get that $p_0 = 0$ and that $q_0 = 0$. Thus $q \sim p \oplus 0 \sim p$.
\end{proof}

\subsection{Properly infinite projections}

\begin{definition}
Let $V$ be an absolute matrix order unit space and let $p \in \mathcal{OP}_n(V) \setminus \lbrace 0 \rbrace$ for some $n \in \mathbb{N}.$ We say that $p$ is properly infinte, if there exist $r,s \in \mathcal{OP}_n(V)$ such that $r+s \leq p$ and $r \sim p \sim s.$ 
\end{definition}
\begin{proposition}
Let $V$ be an absolute matrix order unit space and let $p, q \in \mathcal{OP}_n(V) \setminus \lbrace 0 \rbrace$ for some $n \in \mathbb{N}$. If $p$ and $q$ both are properly infinite and if $p \perp q$, then $p+q$ is also properly infinite. 
\end{proposition}
\begin{proof}
Assume that $p$ and q both are properly infinite. Then there exist $r_1, r_2, s_1, s_2 \in \mathcal{OP}_n(V)$ such that $(r_1 + s_1) \leq p$, $(r_2 + s_2) \leq q$, $r_1 \sim p \sim s_1$ and $r_2 \sim q \sim s_2$. Since $p \perp q$, we have 
$r_1 \perp r_2$ and $s_1 \perp s_2$. Now by Proposition \ref{15}(2), $(r_1 + r_2) \sim (p + q) \sim (s_1 + s_2)$. Also, $(r_1 + r_2) + (s_1 + s_2) \leq p + q$ so that $p + q$ is properly infinite. 
\end{proof}
\begin{theorem}\label{22}
	Let $V$ be an absolute matrix order unit space satisfying condition (T) and let $p \in \mathcal{OP}_\infty(V) \setminus \lbrace 0 \rbrace$. Then the following statements are equivalent:
	\begin{enumerate}
		\item $p$ is properly infinite; 
		\item $p \oplus 0$ is properly infinite;
		\item $0 \oplus p$ is properly infinite; 
	\end{enumerate}
	
Moreover, if $V$ also satisfy (H), then the above statements are also equivalent to:	
	
	\begin{enumerate}
	\item[(4)] $p \oplus p \preceq p$.
	\end{enumerate}
\end{theorem} 
\begin{proof}
	It follows from the definition that (1) implies (2) and (3).  
	
	(2) implies (1): Let $p \oplus 0$ be properly infinite. Then there exist $s_1, s_2 \in \mathcal{OP}_{\infty}(V)$ with $s_1, s_2 \leq p \oplus 0$ such that $s_1 \sim p \oplus 0 \sim s_2$. Since $s_1, s_2 \leq p \oplus 0$, there exist $r_1, r_2 \in \mathcal{OP}_{\infty}(V)$ with $r_1, r_2 \leq p$ such that $s_1 = r_1 \oplus 0$ and $s_2 = r_2 \oplus 0$. Thus $r_1 \oplus 0 \sim p \oplus 0 \sim r_2 \oplus 0$. Now by Proposition \ref{15}(1), we get that $r_1 \sim p \sim r_2$. Thus $p$ is also properly infinite.
	
	Similarly, we can show that (3) implies (1) as well.
	
	Now suppose that $V$ also satisfy (H).
	
	(1) implies (4): Assume that $p$ is properly infinite. Then there exist $r,s \in \mathcal{OP}_\infty(V)$ such that $r+s\leq p$ and $r \sim p \sim s.$ Since $r \perp s,$ by Proposition \ref{15}, we get that $(r+s) \sim p\oplus p.$ Hence $p \oplus p \preceq p.$
	
	(4) implies (1): Finally assume that $p\oplus p \preceq p$. Then by Proposition \ref{17}, we have $p \sim (p \oplus p) \oplus p_0$ for some $p_0 \in \mathcal{OP}_\infty(V)$. As 
	$$(p \oplus 0) \oplus 0 \leq (p \oplus p) \oplus p_0,$$
	$$(0 \oplus p) \oplus 0 \leq (p \oplus p) \oplus p_0$$
	and 
	$$(p \oplus 0) \oplus 0 \perp (0 \oplus p) \oplus 0,$$ 
	by Lemma \ref{13}, there exist $p_1,p_2 \in \mathcal{OP}_\infty(V)$ with $p_1, p_2 \leq p$, and $p_1 \perp p_2$ such that $p_1 \sim (p \oplus 0) \oplus 0$ and $p_2 \sim (0 \oplus p) \oplus 0$.  Thus $p_1 + p_2 \le p$ and $p_1 \sim p \sim p_2$ for $(p \oplus 0) \oplus 0 \sim p \sim (0 \oplus p) \oplus 0$ so that $p$ is properly infinite. 
\end{proof}
\begin{remark} 
	Let $V$ be an absolute matrix order unit space satisfying conditions (T) and (H) and let $p \in \mathcal{OP}_n(V)\setminus \lbrace 0\rbrace$ for some $n \in \mathbb{N}.$ 
	\begin{enumerate}
		\item Then $p$ is infinite whenever it is properly infinite. 
		\item If $q \in \mathcal{OP}_m(V)\setminus \lbrace 0\rbrace$ and if both $p,q$ are properly infinite, then so is $p \oplus q$. 
	\end{enumerate}
\end{remark}
\begin{proposition}\label{23}
Let $V$ be an absolute matrix order unit space satisfying conditions (T) and (H) and let $p, q \in \mathcal{OP}_\infty(V)\setminus \lbrace 0\rbrace$ be such that $p\preceq q \preceq p$. If p is properly infinite, then so is $q$.  
\end{proposition}
\begin{proof}
As $p\preceq q \preceq p$, by Proposition \ref{17}, we have $q \sim p \oplus p_0$ and $p \sim q\oplus q_0$ for some $p_0, q_0 \in \mathcal{OP}_\infty(V)$. Now assume that $p$ is properly infinite. Then by Theorem \ref{22}, $p \oplus p \preceq p$. Thus by Proposition \ref{17}, there exists $r \in \mathcal{OP}_\infty(V)$ such that $p \sim (p \oplus p) \oplus r$. A repeated use of Proposition \ref{15} yields that 
\begin{eqnarray*} 
q &\sim& p\oplus p_0 \sim ((p \oplus p) \oplus r) \oplus p_0 \\
&\sim& (((q\oplus q_0) \oplus (q \oplus q_0)) \oplus r) \oplus p_0 \\
&\sim& (q \oplus q) \oplus (q_0 \oplus q_0 \oplus r \oplus p_0).
\end{eqnarray*}
Again applying Proposition \ref{17} and Theorem \ref{22}, we may conclude that $q$ is properly infinite.
\end{proof} 
\begin{remark}
	It follows from the proof of Proposition \ref{23} that if $p\sim q$ and if $p$ is properly infinite, then so is $q$. 
\end{remark}

\end{document}